\documentclass[12pt]{article}
\usepackage{amsmath,amssymb,amsbsy,amsfonts,amsthm}
\usepackage{url}

\begin{document}

\newtheorem{theorem}{Theorem}
\newtheorem{lemma}[theorem]{Lemma}
\newtheorem{algol}{Algorithm}
\newtheorem{cor}[theorem]{Corollary}
\newtheorem{prop}[theorem]{Proposition}

\newcommand{\comm}[1]{\marginpar{%
\vskip-\baselineskip 
\raggedright\footnotesize
\itshape\hrule\smallskip#1\par\smallskip\hrule}}

\def\cA{{\mathcal A}}
\def\cB{{\mathcal B}}
\def\cC{{\mathcal C}}
\def\cD{{\mathcal D}}
\def\cE{{\mathcal E}}
\def\cF{{\mathcal F}}
\def\cG{{\mathcal G}}
\def\cH{{\mathcal H}}
\def\cI{{\mathcal I}}
\def\cJ{{\mathcal J}}
\def\cK{{\mathcal K}}
\def\cL{{\mathcal L}}
\def\cM{{\mathcal M}}
\def\cN{{\mathcal N}}
\def\cO{{\mathcal O}}
\def\cP{{\mathcal P}}
\def\cQ{{\mathcal Q}}
\def\cR{{\mathcal R}}
\def\cS{{\mathcal S}}
\def\cT{{\mathcal T}}
\def\cU{{\mathcal U}}
\def\cV{{\mathcal V}}
\def\cW{{\mathcal W}}
\def\cX{{\mathcal X}}
\def\cY{{\mathcal Y}}
\def\cZ{{\mathcal Z}}

\def\C{\mathbb{C}}
\def\F{\mathbb{F}}
\def\K{\mathbb{K}}
\def\Z{\mathbb{Z}}
\def\R{\mathbb{R}}
\def\Q{\mathbb{Q}}
\def\N{\mathbb{N}}
\def\M{\textsf{M}}

\def\({\left(}
\def\){\right)}
\def\[{\left[}
\def\]{\right]}
\def\<{\langle}
\def\>{\rangle}

\def\e{e}

\def\eq{\e_q}
\def\fS{{\mathfrak S}}

\def\lcm{{\mathrm{lcm}}\,}

\def\fl#1{\left\lfloor#1\right\rfloor}
\def\rf#1{\left\lceil#1\right\rceil}
\def\mand{\qquad\mbox{and}\qquad}

\def\jt{\tilde\jmath}
\def\ellmax{\ell_{\rm max}}
\def\llog{\log\log}


\title{On the Distribution of\\Atkin and Elkies Primes}

\author{
{\sc Igor E.~Shparlinski} \\
{Department of Computing}\\
{Macquarie University} \\
{Sydney, NSW 2109, Australia} \\
{\tt igor.shparlinski@mq.edu.au}
\and 
{\sc Andrew V. Sutherland} \\
{Department of Mathematics}\\
{Massachusetts Institute of Technology} \\
{Cambridge, Massachusetts 02139, USA} \\
{\tt drew@math.mit.edu}}
\date{December 19, 2011}
\maketitle

\begin{abstract} Given an elliptic curve $E$
over a finite field $\F_q$ of $q$ elements, we say that an odd prime 
$\ell \nmid q$ is an Elkies prime for $E$ if $t_E^2 - 4q$ is a square
modulo~$\ell$, where $t_E = q+1 - \#E(\F_q)$  and $\#E(\F_q)$ is the number of $\F_q$-rational
points on $E$; otherwise $\ell$ is called an Atkin prime.
We show that there are 
asymptotically the same number of 
Atkin and Elkies primes $\ell < L$ on average over all curves $E$ over $\F_q$,
provided that $L \ge  (\log q)^\varepsilon$ for any fixed $\varepsilon>0$
and a sufficiently large $q$. 
We use this result to design and analyse a fast algorithm to 
generate random elliptic curves with $\#E(\F_p)$ prime, where $p$ varies uniformly over primes in a given interval $[x,2x]$.

\end{abstract}


\section{Introduction}  

Let $\F_q$ be a finite field of $q$ elements.
For an elliptic curve $E$ over $\F_q$ we denote by $\#E(\F_q)$  the number of
$\F_q$-rational points on $E$ and define the {\it trace of Frobenius\/}
$t_E = q+1 - \#E(\F_q)$; see~\cite{ACDFLNV,Silv} for a background on 
elliptic curves.
We say that an odd prime 
$\ell \nmid q$ is an  {\it Elkies prime\/} for $E$ if $t_E^2 - 4q$ is a quadratic residue
modulo $\ell$; otherwise $\ell\nmid q$ is called an {\it  Atkin prime\/}. 
For any elliptic curve over a finite field, one expects about the same number of  
Atkin and Elkies primes $\ell < L$ as $L\to \infty$.

These primes play a key role in the {\it Schoof-Elkies-Atkin (SEA) algorithm\/}, see~\cite[\S17.2.2 and \S17.2.5]{ACDFLNV}, and their distribution affects the performance of this
algorithm in a rather dramatic way. Thus we define
$N_a(E;L)$ and $N_e(E;L)$ as the number of Atkin and Elkies primes $\ell$ in the dyadic
interval $\ell \in [L, 2L]$ for 
an elliptic curve $E$ over $\F_q$, respectively. 
We clearly have
$$
N_a(E;L) + N_e(E;L) =  \pi(2L)-\pi(L) + O\(1\),
$$
where $\pi(L)$ denotes the number of primes $\ell < L$, and one expects that
\begin{equation}
\label{eq:NaNe}
N_a(E;L) \sim N_e(E;L)  \sim   \frac12 \(\pi(2L)-\pi(L)\),
\end{equation}
as $L\to \infty$. 

Under the Generalised Riemann Hypothesis (GRH), using the bound 
of quadratic characters over primes, 
it has been noted by Galbraith and Satoh
that~\eqref{eq:NaNe}
holds for $L \ge (\log q)^{2 + \varepsilon}$ for any fixed $\varepsilon>0$
and $q \to \infty$; see~\cite[App.~A]{Sat}, and also~\cite[Prop.~5.25]{IwKow}
or~\cite[Ex.~5.a in \S13.1]{MonVau}. 
However, the unconditional results are much weaker 
and essentially rely on our knowledge of the distribution of 
primes in arithmetic progressions; see~\cite[\S5.9]{IwKow}
or~\cite[Ch.~4 and 11]{MonVau}. 

Here, we study the values of $N_a(E;L)$ and $N_e(E;L)$ on average over 
all elliptic curves $E$ over $\F_q$.  Let $\cE_q$ be any set of  representative of 
all isomorphism classes elliptic curves over $\F_q$.

\begin{theorem} 
\label{thm:AvElkies} For any integer $\nu \ge 1$, we have
\begin{equation*}
\begin{split}
 \frac{1}{\# \cE_q}  \sum_{E \in \cE_q}& \left|N_*(E;L) - \frac12 \(\pi(2L)-\pi(L)\)\right|^{2\nu}\\
 & =O\(  \pi(2L)^{\nu} \log q (\log \log q)^2   + 
  \pi(2L)^{2\nu}    q^{-1/2}  L^{\nu} \log L\),
  \end{split}
\end{equation*}
 where $N_*(E;L)$ is either $N_a(E;L)$ or $N_e(E;L)$. 
\end{theorem} 

For an appropriate choice of $\nu$ we obtain from Theorem~\ref{thm:AvElkies} 
a nontrivial result in the range
$$
(\log q)^{\varepsilon} 
\le L \le q^{1/2} (\log q)^{-1/2-\varepsilon},
$$
for any fixed $\varepsilon> 0$ and all sufficiently 
large $q$. This range includes values of~$L$ that are much smaller than those addressed by
the result of Galbraith and Satoh for any particular elliptic curve,
even under the GRH.

In many applications it is more convenient to consider curves given 
by the family of short  Weierstra\ss\  equations
\begin{equation}
\label{eq:Eab}
E_{a,b}\colon\quad Y^2=X^3+aX+b,
\end{equation}
where $a$ and $b$ run through $\F_q$, with $\gcd(q,6)=1$, and satisfy $4a^3+27b^2\ne 0$. 
Since there are $O(p)$ pairs $(a,b)\in \F_p^2$ for which $E_{a,b}$ lies in a 
given isomorphism class, we easily derive from Theorem~\ref{thm:AvElkies} the following corollary. 

\begin{cor} 
\label{cor:AvElkies-ab} For any real $\varepsilon > 0$ and integer $C \ge 1$, 
for a sufficiently large prime $p$ and $L \ge (\log p)^{\varepsilon}$
there are at most $p^2(\log p)^{-C}$ pairs
$(a,b)\in \F_p^2$ for which  $4a^3+27b^2\ne 0$ and
$$
N_*(E;L) < \frac13 (\pi(2L)-\pi(L)) ,
$$
 where $N_*(E;L)$ is either $N_a(E;L)$ or $N_e(E;L)$. 
\end{cor} 

As an application of Corollary~\ref{cor:AvElkies-ab}, in Section~\ref{sec:alg} we present Algorithm~\ref{alg:Prime card},
which efficiently generates a random elliptic curve of prime order.
Given an integer~$x>3$, we seek a uniformly random element of the set $T(x)$ of all triples $(p,a,b)$, where $p$ is a prime in the interval $[x,2x]$, while~$a$ and~$b$ are elements of $\F_p$ for which the elliptic curve 
$E_{a,b}$ in~\eqref{eq:Eab}
has a prime number of $\F_p$-rational points. 
This problem arises in cryptographic applications of elliptic curves, where one typically requires a curve with prime (or near prime) order, but wishes to choose a curve that is otherwise as generic as possible. 

We show that the output and complexity of Algorithm~\ref{alg:Prime card} (see Section~\ref{sec:alg}) 
satisfy the following:

\begin{theorem}\label{thm:PrimeCard}
Given a real number $x> 3$, Algorithm~\ref{alg:Prime card} outputs a prime $p\in[x,2x]$, two elements $a,b\in\F_p$, and $N=\#E_{a,b}(\F_p)$,
where $N$ is prime and $(p,a,b)$ is uniformly distributed over $T(x)$.
Assuming the GRH, the expected running time of Algorithm~\ref{alg:Prime card} is $O((\log x)^5(\log\log x)^3\log\log\log x)$.
\end{theorem}

\section{Preparations}
\label{sec:prep}

We recall
the notations $U = O(V)$, $V = \Omega(U)$, $U \ll V$ and  $V \gg U$, which are all
equivalent to the statement that the inequality $|U| \le c\,V$ holds asymptotically,
with some constant $c> 0$. We also write $U = \widetilde{O}(V)$ to
indicate that $|U| \le V (\log V)^{O(1)}$. 
Throughout the paper, any implied constants
in these symbols  may occasionally depend, where obvious, on
the integer parameter $\nu \ge 1$ and the real parameter $\varepsilon>0$,  
and are absolute otherwise.
We always assume  that $\ell$ runs though the prime values. 

Let us first recall some known facts about elliptic curves, which 
are conveniently summarised by Lenstra~\cite{Len}.
In particular, we need the following well-know asymptotic estimate on 
the cardinality of $\# \cE_q$; see~\cite[\S1.4]{Len} for $\gcd(q,6)=1$, 
\cite[Thm.~3.18]{HMV} for $2\mid q$, and~\cite{Jeo} for $3\mid q$.

\begin{lemma} 
\label{lem:Ep}  We have
$$
\# \cE_q = 2q + O(1).
$$
\end{lemma} 

Furthermore, let $f_q(t)$ be the number of isomorphism classes of curves $E$
over $\F_q$ with $t_E = t$.  Lenstra gives in~\cite[Prop.~1.9]{Len}
the following upper bound on $f_q(t)$, which we formulate together 
with the  Hasse estimate on  possible values of $t$; see~\cite[Prop.~1.5]{Len}
or~\cite{ACDFLNV,Silv}.

\begin{lemma} 
\label{lem:fpt}  We have
$$
f_q(t) \ll   \left\{
\begin{array}{ll}
0, & \quad \text{ if   $|t|>2q^{1/2}$},\\
q^{1/2}\log q (\log \log q)^2, & \quad \text{ if  $|t|\le 2q^{1/2}$}.
\end{array}
\right.
$$
\end{lemma} 

We also need some results on multiplicative character sums. 
More precisely, we concentrate on the sums of Jacobi symbols
$(a/b)$; see~\cite[\S~3.5]{IwKow}.  Let us first consider complete sums.

\begin{lemma} 
\label{lem:SumCompl}  For any integer  $a$ and a product 
$m = \ell_1 \ldots \ell_s$ 
of $s$ distinct odd primes $ \ell_1, \ldots,\ell_s$ 
with $\gcd(a,m)=1$ we have
$$
\left|\sum_{t=0}^{m-1} \(\frac{t^2 - a}{m}\) \right| = 1.
$$
\end{lemma} 

\begin{proof}
 We use the following special case of the well-known  identity
 for sums of Legendre symbols with quadratic polynomials
see~\cite[Thm.~5.48]{LN}:
$$
\sum_{t=0}^{\ell-1} \(\frac{t^2 - a}{\ell}\)  =  - \(\frac{1}{\ell}\) 
$$
for any prime $\ell \nmid a$. Applying the multiplicativity of 
complete character sums,  see~\cite[Eq.~12.21]{IwKow},
completes the proof. 
\end{proof}

The following estimate is a slight 
generalisation of~\cite[Lem.~2.2]{LuSh}.

\begin{lemma} 
\label{lem:SumIcompl}  For any integers $a$ and $T\ge 1$  and a product 
$m = \ell_1 \ldots \ell_s$ 
of $s\ge 0$ distinct odd primes $ \ell_1, \ldots,\ell_s$ 
with $\gcd(a,m)=1$ we have
$$
\sum_{|t|\le T} \(\frac{t^2 - a}{m}\) \ll T/m + C^s m^{1/2} \log m,
$$
for some absolute constant $C \ge 1$.
\end{lemma} 

\begin{proof} The result is trivial when $s= 0$, that is, when $m=1$

For $s \ge 1$,  as in~\cite{LuSh},  we note that the Weil bound 
applied to the mixed sums of additive and 
multiplicative characters with polynomials, of the 
type given in~\cite[Eq.~11.43]{IwKow}, and 
the multiplicativity of complete character sums,
see~\cite[Eq.~12.21]{IwKow},  imply that 
$$
\sum_{t= 1}^{m} \(\frac{t^2 - a}{m}\) \exp\(2\pi i
\frac{\lambda t}{m}\) \ll C^sm^{1/2}  
$$
holds for any integer $\lambda$ and
some absolute constant $C \ge 1$. 
Using the standard reduction between complete and incomplete sums
(see~\cite[\S~12.2]{IwKow}), we derive that for any 
integer $K$ and any positive integer $L\le m$ we have 
\begin{equation}
\label{eq:Incomp}
\sum_{t= K+1}^{K+L} \(\frac{t^2 - a}{m}\) \exp\(2\pi i
\frac{\lambda t}{m}\) \ll C^sm^{1/2} \log m.
\end{equation}

Separating the summation range over $t$ into $O(T/m)$ intervals of length~$m$ 
(and using Lemma~\ref{lem:SumCompl} for the sums over these 
intervals) 
and  at most one interval of length
$m$ (and using~\eqref{eq:Incomp} for the sums over these 
intervals),  we obtain the desired  result.
\end{proof}

Finally, for any integer $n$ we denote by $\omega_L(n)$
the number of primes in the interval $[L,2L]$ that divide $n$.

\begin{lemma} 
\label{lem:PrimeDiv}  For $L \ge 3$ and any integer $\nu\ge 1$, we have
$$
\sum_{|t|< T} \omega_L^{\nu}(t^2 - a) \ll \frac{T}{\log  L} + \frac{L^\nu}{(\log  L)^\nu}.
$$
\end{lemma} 

\begin{proof}  We have
\begin{equation*}
\begin{split}
\sum_{|t|< T} \omega_L^{\nu}(t^2 - a) =  
\sum_{|t|< T}
\(\sum_{\substack{L \le \ell \le 2L\\ \ell \mid t^2 -a}} 1\)^\nu 
=
\sum_{L \le \ell_1, \ldots \ell_\nu \le 2 L}
\sum_{\substack{|t| < T\\ 
\lcm[\ell_1, \ldots, \ell_\nu] \mid t^2 -a}} 1.
\end{split}
\end{equation*}
By the Chinese remainder theorem, for any squarefree $m\ge 1$ we have 
$$
\sum_{\substack{|t| < T\\ 
m \mid t^2 -a}} 1 \ll 2^j(T/m+1),
$$
where $j = \omega(m)$ counts the prime divisors of $m$. 
Now, for each $j=1, \ldots, \nu$ we collect together 
the terms such that among $\ell_1, \ldots \ell_\nu < L$,
only are $j$ distinct. 
We then obtain
\begin{equation*}
\begin{split}
\sum_{|t|< T} \omega_L^{\nu}(t^2 - a) & \ll    
\sum_{j=1}^\nu  \sum_{L \le \ell_1, \ldots, \ell_j\le  2 L} 
\(\frac{T}{\ell_1  \ldots \ell_j} + 1\)\\
& \le 
\sum_{j=1}^\nu \(T \(\sum_{L \le \ell \le  2L} 
 \frac{1}{L}\)^j + \pi(2L)^j\).
\end{split}
\end{equation*}
Applying the Prime Number Theorem completes the proof.
\end{proof}

\section{Proof of Theorem~\ref{thm:AvElkies}}

Clearly, we have 
$$
N_a(E;L) - N_e(E;L)  = \sum_{L \le \ell \le 2L} \(\frac{t_E^2 - 4q}{\ell}\)
+ O\(\omega_L(t_E^2-4q) +1\),
$$
where, as before, $\omega_L(n)$   denotes the number of primes  
$\ell\in[L,2L]$ with $\ell \mid n$.

Therefore,  
\begin{equation}
\label{eq:UV}
 \frac{1}{\# \cE_q} \sum_{E \in \cE_q} \left|N_*(E;L) - \frac12 \(\pi(2L)-\pi(L)\)\right|^{2 \nu}
\ll  \frac{1}{\# \cE_q} U_\nu +    \frac{1}{\# \cE_q} V_\nu + 1, 
\end{equation}
where as before $N_*(E;L)$ is either $N_a(E;L)$ or $N_e(E;L)$ and 
$$
U_\nu = 
\sum_{E \in \cE_q} \left| \sum_{L \le \ell \le 2L} \(\frac{t_E^2 - 4q}{\ell}\)\right|^{2 \nu}
\mand 
V_\nu = 
\sum_{E \in \cE_q}  \omega_L(4q-t_E^2)^{2\nu}.
$$

By Lemma~\ref{lem:fpt}
\begin{equation}
\label{eq:Wfpt}
\begin{split}
 U_\nu & =   \sum_{|t| < 2q^{1/2}} f_q(t)  
\left| \sum_{L \le \ell \le 2L} \(\frac{t^2 - 4q}{\ell}\)\right|^{2\nu}\\
& \ll q^{1/2}\log q (\log \log q)^2
 \sum_{|t| < 2q^{1/2}} 
\left| \sum_{L \le \ell \le 2L} \(\frac{t^2 - 4q}{\ell}\)\right|^{2\nu}, 
\end{split}
\end{equation}
where $f_q(t)$ is defined as in Section~\ref{sec:prep}. Furthermore, 
$$
\sum_{|t| < 2q^{1/2}} 
\left| \sum_{L \le \ell \le 2L} \(\frac{t^2 - 4q}{\ell}\)\right|^{2\nu}  
= \sum_{3 \le \ell_1,\ldots,\ell_{2\nu} \le L}
 \sum_{|t| < 2q^{1/2}}  \(\frac{t^2 - 4q}{\ell_1\ldots\ell_{2\nu}}\).
$$

For every $j=0, \ldots \nu$ let $\cQ_j$ be the set 
of $2\nu$ tuples  $(\ell_1,\ldots,\ell_{2\nu})$ of primes with 
$L \le \ell_1,\ldots,\ell_{2\nu} \le 2L$ such that the product
$r=\ell_1 \ldots \ell_{2\nu}$ is of the form $r = k^2 m$,
with $m$ squarefree and $k$ the product of $j$ 
primes.

For the cardinalities of these sets we clearly have
$$
\# \cQ_j \ll \(\pi(2L)-\pi(L)\)^{2\nu - j} \ll \frac{L^{2\nu -j}}{(\log L)^{2\nu -j}}. 
$$

Using Lemma~\ref{lem:SumIcompl} for   $(\ell_1,\ldots,\ell_{2\nu})\in  \cQ_j$,
$j=0, \ldots, \nu$, 
we obtain
\begin{equation*}
\begin{split}
\sum_{|t| < 2q^{1/2}} 
\left| \sum_{L \le \ell \le 2L} \(\frac{t^2 - 4q}{\ell}\)\right|^{2\nu} 
   & \ll    
\sum_{j=0}^{\nu} \# \cQ_j ( q^{1/2}/L^{2\nu - 2j} + L^{\nu - j}\log L)\\
  & \ll      
\sum_{j=0}^{\nu}   \( q^{1/2} \frac{L^{j}}{(\log L)^{2\nu -j}} +  
\frac{L^{3\nu -2j}}{(\log L)^{2\nu -j-1}}\)\\
   &  \ll
  q^{1/2} \frac{L^{\nu}}{(\log L)^{\nu}} +  
\frac{L^{3\nu}}{(\log L)^{2\nu -1}}.
\end{split}
\end{equation*}

Inserting this bound in~\eqref{eq:Wfpt}, we obtain
\begin{equation}
\label{eq:U bound}
 U_\nu\ll q^{1/2}\log q (\log \log q)^2\( q^{1/2} \frac{L^{\nu}}{(\log L)^{\nu}} +  
\frac{L^{3\nu}}{(\log L)^{2\nu -1}}\).
\end{equation}

Finally, by Lemma~\ref{lem:PrimeDiv} we have
\begin{equation}
\label{eq:V bound}
V_\nu\ll q^{1/2}\log q (\log \log q)^2\( \frac{q^{1/2}}{\log  L} + 
\frac{L^{2\nu}}{(\log  L)^{2\nu}}\).
\end{equation}

Substituting~\eqref{eq:U bound} and~\eqref{eq:V bound} in~\eqref{eq:UV} and recalling
Lemma~\ref{lem:Ep}, we conclude the proof.

\section{Point Counting on Random Elliptic Curves}

We now consider the problem of generating a random elliptic curve whose group of $\F_p$-rational points has prime order.
 One approach is to fix the prime~$p$, and then count points on randomly generated elliptic curves over $\F_p$ until a curve with prime order is found.  Using the SEA point-counting algorithm, this procedure heuristically has an expected running time of $\widetilde{O}(n^5)$, where $n=\log p$.  However, for a fixed prime~$p$, we cannot hope to prove even a polynomial time bound, because
even under the GRH 
the Hasse interval $[p-2\sqrt{p},p+2\sqrt{p}]$ is too narrow to permit a useful lower bound on the number of primes it contains.  Thus we let $p$ vary over an interval $[x,2x]$, which at least makes a polynomial-time bound feasible; see~\cite{Kob}.

A second obstacle to obtaining an $\widetilde{O}(n^5)$ expected time bound is that the expected running time of the SEA algorithm is not known to be polynomial in $n$, unless we assume the GRH.  Even with the GRH, the expected running time of the SEA algorithm on any particular curve is only bounded by $\widetilde{O}(n^5)$, yielding an $\widetilde{O}(n^6)$ bound overall.  However, for randomly generated curves, Theorem~\ref{thm:AvElkies} yields a tighter bound, on average, allowing us to prove an $\widetilde{O}(n^5)$ bound on the expected time to find a curve of prime order, under the GRH.

We first present an algorithm that attempts to count the points on the elliptic curve $E_{a,b}$ modulo~$p$, using a simplified version of the SEA algorithm that relies only on Elkies primes.  In the course of doing so, the algorithm may discover that $p$ is composite
(using the Miller-Rabin algorithm~\cite{Rab}), 
or that the curve $E_{a,b}$ is singular modulo $p$, and in either case it outputs $0$; otherwise, it returns a positive integer $N$ in the Hasse interval  $[p-2\sqrt{p},p+2\sqrt{p}]$.  If $p$ is in fact prime (and $E_{a,b}$ is not singular), then $N$ is equal to $\#E_{a,b}(\F_p)$.

\begin{algol}
\label{alg:SEA count} Point-counting modulo $p$ using Elkies primes.
\begin{description}
\item[Input:] An integer $p>3$ and integers $a,b\in[0,p-1]$.
\item[Output:] A positive integer $N\in[p-2\sqrt{p},p+2\sqrt{p}]$
with $\#E_{a,b}(\F_p)= N$ if~$p$ is prime and $4a^3+27b^2\not\equiv 0\bmod p$, and $0$ otherwise.
\end{description}
\end{algol}

\begin{enumerate}
\item In parallel to the steps below, repeatedly test $p$ for compositeness using the Miller-Rabin algorithm~\cite{Rab}.
If at any point $p$ is found to be composite then output $0$ and terminate.
\item If $\gcd(4a^3+27b^2,p)\ne 1$ then output $0$ and terminate.  Otherwise, 
set $j  \leftarrow 1728\frac{4a^3}{4a^3+27b^2}\bmod p$.
\item Test whether $E=E_{a,b}\bmod p$ is supersingular using~\cite[Alg.~2]{Sut1}.\\
If so, then output $p+1$ and terminate.
\item Set $i  \leftarrow  0$, $M  \leftarrow 1$, and for primes $\ell = 2,3,5,\ldots$, do the following:
\begin{enumerate}
\item Compute the modular polynomial $\Phi_\ell(X,Y)$ using~\cite[Alg.~1]{BLS}.
\item Compute $\phi(X)=\Phi_\ell(j,X)$ and $f(X) =\gcd (X^p-X,\phi(x))$ in the ring $(\Z/p\Z)[X]$.  If $\deg f = 0$ then proceed to the next prime $\ell$.
\item Find a root $\jt$ of $f(X)$ modulo $p$.
\item Compute the Elkies polynomial $h(X)$ whose roots are the abscissae of the points in the kernel of the $\ell$-isogeny $\phi$ from $E$ to a curve $\widetilde{E}$ with $j$-invariant $\jt$.\footnote{The special case $(\partial\Phi_\ell/\partial X)(j,\jt)=(\partial\Phi_\ell/\partial Y) (j,\jt) = 0$ must be handled separately, see the proof of Lemma~\ref{lem:SEA count} for details.}
\item Using $h$, determine the integer $\lambda\in[1,\ell-1]$ for which the $p$-power Frobenius action on $\ker\phi$ is equivalent to multiplication by $\lambda$.
If no such $\lambda$ exists then output $0$ and terminate.
\item Set $i \leftarrow i+1$,  $\ell_i\leftarrow \ell$, 
$M \leftarrow M \ell$, and $t_i \leftarrow \lambda + p/\lambda \bmod \ell$.
\item If $M > 4\sqrt{p}$ then proceed to Step~5.  Otherwise, continue Step~4.
\end{enumerate}
\item
Compute the unique integer $t\in[-M,M]$ for which $t\equiv t_i\bmod \ell_i$ for each Elkies prime $\ell_i$.  If $|t|>2\sqrt{p}$ then output 0, otherwise, output $N=p+1-t$.
\end{enumerate}

We note that the algorithm is not in any sense required to be ``correct" when $p$ is composite, it may output either 0 or any integer $N$ in the Hasse interval in this case.  The Miller-Rabin tests begun in Step~1 of Algorithm~\ref{alg:SEA count} are there simply to ensure that composite $p$ are handled efficiently.  This is necessary since the rest of the algorithm, operating on the assumption that $p$ is prime, may run extremely slowly or even fail to terminate if~$p$ is composite.

Assuming that $p$ is prime, the value $j$ computed in Step~2 is the $j$-invariant of the elliptic curve 
$E=E_{a,b}$ over $\F_p$.
The classical modular polynomial $\Phi_\ell$ parametrises pairs of $\ell$-isogenous elliptic curves; the roots of $\Phi_\ell(j(E),X)$ are the $j$-invariants of the curves~$\widetilde{E}$ that are related to $E$ by a cyclic isogeny of degree $\ell$.  There exists such 
an elliptic curve $\widetilde{E}$ 
defined over $\F_p$ precisely when $\ell$ is an Elkies prime for $E$, thus Elkies and Atkin primes are distinguished in Steps~4b and~4c, which attempt to find a root of $\Phi_\ell(j(E),X)$  in $\F_p$.  Steps~4c-4f then apply the standard SEA procedure for computing the trace of Frobenius modulo an Elkies prime $\ell$, as described by Schoof in~\cite{Sch}.

We now consider the complexity of Algorithm~\ref{alg:SEA count}.  We use the asymptotic bound $O(n\log n\log \log n)$ of Sch\"onhage and Strassen~\cite{SS} to bound the time $\M(n)$ to multiply to $n$-bit integers (see also~\cite{GG}), and note that all of our complexity estimates count bit operations.

\begin{lemma}\label{lem:SEA count}
Let $n=\rf{\log p}$ and assume the GRH.  For composite $p$, the expected running time of Algorithm~\ref{alg:SEA count} is $O(n^2\log n\log\log n)$.
For prime $p$, the average expected running time of Algorithm~\ref{alg:SEA count} over 
integers $a,b\in[0,p-1]$ is $O(n^4(\log n)^2\log \log n)$.
\end{lemma}

\begin{proof}
We expect to detect a composite $p$ using $O(1)$ Miller-Rabin tests, each of which has complexity $O(n\M(n))=O(n^2\log n\log\log n)$, the time to perform an exponentiation modulo $p$.  This proves the first claim.

We now assume $p$ is prime.  The complexity of Step~2 is $O(\M(n)\log n)$, and Step~3 runs in  $O(n^3\log n\log\log n)$ expected time; see~\cite[Prop.~4]{Sut1}.

Let $m$ be the largest prime $\ell$ used in Step~4.  We have 
$\log M \ge n/2$ thus by 
the Prime Number Theorem, $m \gg n$.  Ignoring constant factors, we may use $m$ as an upper bound on both $\ell$ and $n$.  Table~\ref{tab:step4} estimates the costs of Steps~4a-4f in terms of $\ell$ and $n$, and also gives bounds in terms of $m$.  We use standard asymptotic bounds on the complexity of (fast) arithmetic operations in $\Z/p\Z$ and $\Z/p\Z[X]$, all of which can be found in~\cite{GG}.\footnote{Some of these bounds can be improved by using Kronecker substitution to multiply polynomials in $\Z/p\Z[X]$, but this does not change the overall complexity.}

In Step~4a we use the isogeny volcano algorithm of~\cite{BLS} to compute the modular polynomial $\Phi_\ell$, and it is here that we need to assume the GRH.  
In the complexity bound for Step~4d we include the cost of computing and evaluating various partial derivatives of $\Phi_\ell$ modulo $p$, and use Elkies' algorithm to compute the kernel polynomial $h(X)$; 
see~\cite{Elk} and~\cite[Ch.~25]{Gal} for details, and~\cite{BMSS} for further optimizations.  
In  the complexity bound for Step~4e,
the first term bounds the time to compute the action of Frobenius on $\ker \phi$ (this involves computing $X^p$ and $Y^p$ modulo $h$ and $E_{a,b}$), while the second term bounds the time to compute the action of multiplication by $\lambda$ on $\ker \phi$ for every integer $\lambda$ in $[1,\ell-1]$; see~\cite{GM} for details and optimizations. 

\begin{table}
\begin{center}
\caption{Complexity bounds for Step~4 of Algorithm~\ref{alg:SEA count}}\label{tab:step4}
\medskip

\begin{tabular}{clll}
step & result  & expected time $O(\cdots)$ & in terms of $m$\\\hline
(a) & $\Phi_\ell(X,Y)$ & $\ell^3(\log \ell)^3\log\log \ell$ & $m^3(\log m)^3\llog m$\\
(b) & $\phi(X)$ & $\ell^2\M(\ell\log\ell+n)$ & $m^3(\log m)^2\llog m$\\
     & $X^p\bmod \phi$ & $n\M(\ell)\M(n)$ & $m^3(\log m)^2(\llog m)^2$\\
     & $f(X)$ & $\M(\ell)\M(n)\log\ell+\ell\M(n)\log n$ & $m^2(\log m)^3(\llog m)^2$\\
(c) & $\jt$ & $\M(\ell)\M(n)n$ & $m^3(\log m)^2(\llog m)^2$\\
(d) & $h(X)$ & $\ell^2\M(n)+\M(n)\ell\log n$ & $m^3\log m\llog m$\\
(e) & $\lambda$ & $\M(\ell)\M(n)n+\M(\ell)\M(n)\ell$ & $m^3(\log m)^2(\llog m)^2$\\
(f)  & $t_i$ & $\ell\M(\log \ell)\llog\ell$ & $m\log m(\llog m)^2$\\
     & $M$ & $\M(n+\log \ell)$ & $m\log m\llog m$\\
\end{tabular}
\end{center}
\end{table}

The cost of Steps~4a--4f is dominated by the $O(m^3(\log m)^3\log\log m)$ cost of Step~4a, which also dominates the cost of Steps~2, 3, and~5, the last of which has complexity $O(\M(m)\log m)$.
The number of iterations in Step~4 is at most $\pi(m)=O(m/\log m)$, thus when $p$ is prime, the total 
expected running time of Algorithm~\ref{alg:SEA count} is $O(m^4(\log m)^2\log\log m)$.

To address the special case $(\partial\Phi_\ell/\partial X)(j,\jt)=(\partial\Phi_\ell/\partial Y) (j,\jt) = 0$, we note that, as explained by Schoof in~\cite[pp.~248--249]{Sch}, there are then only $O(\ell^2)$ possible values for $N$.  For $p>229$, only one of these candidates satisfies Mestre's theorem~\cite[Thm.~3.2]{Sch}.  By multiplying random points on $E_{a,b}(\F_p)$ and its quadratic twist by each of the candidate values for $N$, we can uniquely determine $N$ in $O(\ell^2n\M(n))=O(m^4\log m\log\log m)$ expected time, which is dominated by the bound we derived above (and for $p\le 229$ we can simply enumerate the elements of $E_{a,b}(\F_p)$ by brute force).

We now notice that by Corollary~\eqref{cor:AvElkies-ab} and the Prime Number Theorem we have $m \ll n$
for all but $O(p^2 n^{-3})$ pairs $(a,b)\in \F_p^2$ for which, by
the result of Galbraith and Satoh~\cite[App.~A]{Sat}, we have $m \ll n^3$. 

Thus if we average over all integers of $a,b$ in $[0,p-1]$ for a fixed prime~$p$, then the expected 
value of $m$ is $O(n)$, which completes the proof. \end{proof}

\section{Proof of Theorem~\ref{thm:PrimeCard}}
\label{sec:alg}

The proof is based on the analysis of the following procedure:

\begin{algol}
\label{alg:Prime card} Generation of a random elliptic curve with a prime number of rational points over a finite field.
\begin{description}
\item[Input:] A real $x>3$.
\item[Output:] A prime $p\in[x,2x]$, $a,b\in\F_p$, and  $N=\#E_{a,b}(\F_p)$ prime.
\end{description}
\end{algol}

\begin{enumerate}
\item Pick a uniformly random integer $p$ in the interval $[x,2x]$.
\item Pick uniformly random integers $a,b\in[0,p-1]$ and apply Algorithm~\ref{alg:SEA count}
 to $E_{a,b}\bmod p$, obtaining $N$.
If Algorithm~\ref{alg:SEA count} finds that $p$ is composite or that $p|(4a^3+27b^2)$ then return to Step~1.
\item Apply $\lceil\log x\rceil$ Miller-Rabin tests to both $p$ and $N$.
If either $p$ or $N$ is found to be composite then return to Step~1.
\item Determine the primality of $p$ and $N$ using a randomized AKS algorithm~\cite{Ber}.
If $N$ and $p$ are both prime, then output $p$, $a$, $b$, and $N$, and terminate.  Otherwise, return to Step~1.
\end{enumerate}

The Miller-Rabin algorithm~\cite{Rab} attempts to prove that a given integer~$p$ is not prime 
(that is, composite) 
via a sequence of independent random tests, each of which detects 
a composite $p$ with probability at least 3/4.  Thus the probability that the algorithm reaches Step~4 when $N$ is composite is less than $1/\log x$.
The primality testing algorithm used in Step~4 is a randomized version of the Agarwal-Kayal-Saks algorithm~\cite{AKS} due to Bernstein~\cite{Ber}, and 
determines whether $N$ is prime or composite in $O(n^{4+\varepsilon})$ expected time, for any $\varepsilon > 0$.

We now put $n=\rf{\log x}$ 
and show that the expected running time of Algorithm~\ref{alg:Prime card} is $O(n^5(\log n)^3\log\log n)$.

Step~2 of Algorithm~\ref{alg:Prime card} calls Algorithm~\ref{alg:SEA count} with parameters chosen uniformly at random from the set $T(x)$ of triples $(p,a,b)$ with $x\le p\le 2x$ and $0\le a,b < p$, which has cardinality $O(x^3)$.
Let $S(x)$ denote the subset of $T(x)$ consisting of those triples $(p,a,b)$ for which both $p$ and $N=\#E_{a,b}(\F_p)$ are prime (and $p\nmid (4a^3+27b^2)$, which we assume throughout).

We first show that the cardinality of $S(x)$ satisfies
\begin{equation}
\label{eq:Card S}
\#S(x) \gg \frac{x^3}{(\log x)^2 \log\log x}.
\end{equation}
By~\cite[Lem.~1]{Kob}, the number of pairs of primes $(p,N)$ with $x\le p\le 2x$ and $p-\sqrt{p}\le N\le p+\sqrt{p}$ is $\Omega(x^{3/2}/(\log x)^2)$.
For each such pair $(p,N)$, the number of pairs $(a,b)$ with $0\le a,b < p$ for which $\#E_{a,b}(\F_p)=N$ is $\frac{1}{2}(p-1)H(D)$, where $D=(p+1-N)^2-4p$, and $H(D)$ denotes the Hurwitz class number, see~\cite[Thm.~14.18]{Cox}. 
 Let $D=vD_0$, where $D_0$ is a fundamental discriminant.
By~\cite[Lem.~9]{Sut2}, we have $H(D)\ge vH(D_0) \ge \frac{1}{3}v h(D_0)$, and the GRH implies 
$$
h(D_0)\gg \sqrt{|D_0|}/\log\log|D_0|,
$$ 
where $h(D_0)$ is the usual class number, by a theorem of Littlewood~\cite{Lit}.
It follows that 
$$
H(D) \gg \sqrt{|D|}/\log\log|D| \gg \sqrt{x}/\log\log x,
$$  
see also comments in~\cite[Section~1.6]{Len}. 
Therefore,  
there are $\Omega(x^{3/2}/\log\log x)$ pairs $(a,b)$ with $\#E_{a,b}(\F_p)=N$ and $\Omega(x^{3/2}/(\log x)^2)$ 
pairs of primes $(p,N)$, which implies~\eqref{eq:Card S}. 

Thus we expect to generate $O((\log x)^2\log\log x)=O(n^2\log n)$ random triples $(p,a,b)$ in order to obtain a triple for which $p$ and $N=\#E_{a,b}(\F_p)$ are both prime.
Once this occurs, the algorithm 
successfully completes Steps~2-5 and terminates.
We now consider the cost of processing each random triple, which we divide into 3 cases.
\begin{enumerate}
\item If $p$ is composite, the expected cost of Step~2 is $O(n^2\log n\log\log n)$, by Lemma~\ref{lem:SEA count}, which also bounds the complexity of Step~3 (assuming it is reached), since we actually expect to discover that $p$ is composite using just $O(1)$ Miller-Rabin tests.
The probability of reaching Step~4 is less than $4^{-\log x}=O(1/x)$, by~\cite{Rab}, which makes the conditional cost of Steps~4 and~5 in this case completely negligible, since they both have expected running times that are polynomial in $\log x$.
\item If $p$ is prime and $N$ is composite, the expected cost of Step~2 given by Lemma~\ref{lem:SEA count} is $O(n^4(\log n)^2\log\log n)$, which dominates the complexity of Step~3 and the conditional cost of Step~4 (which, as in Case~1, we have negligible probability of reaching).
\item If $p$ and $N$ are prime, the expected costs of Steps~2, 3, 4, and~5 are, respectively,  $O(n^4(\log n)^2\log\log n)$, $O(n^3\log n\log\log n)$, $O(n^{4+\varepsilon})$, and $O(n^2\log n\log\log n)$; see~\cite{Ber} for the bound on Step~4.
Thus the total expected cost is $O(n^{4+\varepsilon})$ for any $\varepsilon>0$. 
\end{enumerate}

We now bound the expected running time of Algorithm~\ref{alg:Prime card} by considering how often we expect each case to occur.
We expect to be in Case~1 for $O(n^2\log n)$ triples, each of which takes $O(n^2\log n\log\log n)$ time, yielding a total bound of $O(n^{4+\varepsilon})$.
We expect to be in Case~2 for $O(n\log n)$ triples, each of which takes $O(n^4(\log n)^2\log\log n)$ expected time, yielding a total bound of $O(n^5(\log n)^3\log\log n)$.
Case~3 occurs exactly once, and takes $O(n^{4+\varepsilon})$ expected time.
Case~2 dominates and the theorem follows.

\section{Comments}

The bound in Theorem~\ref{thm:PrimeCard} would be improved by a factor of $\log n$ if one could show 
that $H(D)=\Omega(\sqrt{|D|})$, on average, but we have not attempted to do this (note that the distribution of $D$ is not uniform).
As a practical optimization, one can add an ``early abort" option in Algorithm~\ref{alg:SEA count} that causes the algorithm to terminate if it discovers that $N\equiv 0\bmod \ell$.
Heuristically, this should reduce the running time of Algorithm~\ref{alg:Prime card} by a factor of $\log n$.
Another practical optimization is to reuse the modular polynomials $\Phi_\ell$ that are computed in Algorithm~\ref{alg:SEA count}, which do not depend on the inputs $p$, $a$, and~$b$.
This saves a factor of $\log n$ in the expected running time, but increases the expected space complexity from $O(n^3\log n)$ to $O(n^4\log n)$.
Combining these two optimizations with the assumption that $H(D) \gg \sqrt{|D|}$ on average, yields a heuristic expected running time of $O(n^5\log\log n)$ for Algorithm~\ref{alg:Prime card}.

\section*{Acknowledgement}

During the preparation  I.~E.~Shparlinski was supported in part
by ARC grant DP1092835, and by NRF Grant~CRP2-2007-03, Singapore.
A.~V.~Sutherland received financial support from NSF grant DMS-1115455.

\end{document}